\documentclass[11pt]{article}
\usepackage[T1]{fontenc}
\usepackage{lmodern}
\usepackage{amsmath,amssymb,amsthm,mathtools}
\usepackage{geometry}
\usepackage{microtype}
\usepackage[colorlinks=true,linkcolor=blue,citecolor=blue,urlcolor=blue]{hyperref}
\newtheorem{theorem}{Theorem}[section]
\newtheorem{lemma}[theorem]{Lemma}

\newtheorem{proposition}[theorem]{Proposition}
\newtheorem{remark}[theorem]{Remark}

\DeclareMathOperator{\Inf}{Inf}

\newcommand{\calA}{\mathcal A}
\newcommand{\calB}{\mathcal B}
\newcommand{\calC}{\mathcal C}
\newcommand{\calD}{\mathcal D}
\newcommand{\calE}{\mathcal E}
\newcommand{\calF}{\mathcal F}

\newcommand{\calH}{\mathcal H}
\newcommand{\calK}{\mathcal K}
\newcommand{\calU}{\mathcal U}
\title{Proof of Fishburn's latent-subset conjecture}
\author{Yuxian Dong, Jianxi Mao
\thanks{
	Corresponding author.
    \newline\hspace*{3mm}
    {\it Email addresses:}\quad
    yuxiand@hotmail.com (Yuxian Dong); maojx@dlut.edu.cn (Jianxi Mao)}
}
\date{\footnotesize
 School of Mathematical Sciences, Dalian University of Technology, Dalian 116024, P.R. China\\
}
\begin{document}
\maketitle

\begin{abstract}
Fishburn's latent-subset conjecture, proposed in 1987 and revisited in 1988, asserts that for every dual intersecting family $\mathcal F\subseteq 2^{[n]}$, there exists $i\in[n]$ such that $|\mathcal F^L(i)|\geq|\mathcal F(i)|$.
Here $\mathcal F^L$ is the family of the subsets of members of $\mathcal F$ that do not belong to $\mathcal F$. 
We prove the conjecture using the recent weighted star inequality of Chang, Liu, and Liu.
\end{abstract}

\section{Introduction}
A family $\calF\subseteq2^{[n]}$ is \emph{intersecting} if $A\cap A'\ne\varnothing$ for all $A,A'\in\calF$.
For a family $\calA\subseteq 2^{[n]}$, its family of
\emph{lower latent subsets} is
\[
\calA^L
=\{T\subseteq[n]:T\notin\calA,\ 
  T\subseteq A\text{ for some }A\in\calA\}.
\]
Kleitman and Magnanti~\cite{KleitmanMagnanti} proved that every intersecting family $\calA$ satisfies $|\calA^L|\ge|\calA|$. 

A family $\calF\subseteq 2^{[n]}$ is \emph{dual intersecting} if $A\cap A'\ne\varnothing$ and $A\cup A'\ne[n]$ for all $A,A'\in\calF$. 
Equivalently, both $\calF$ and the family of complements of its members are intersecting.
For $i\in[n]$, the \emph{$i$-star} of $\calF$ is $\calF(i)=\{A\in\calF:i\in A\}$.
Clearly, every star is intersecting.
Write
$$\mathcal F^L(i)=\{T\in\mathcal F^L:i\in T\}$$
for the family of lower latent subsets of $\mathcal F$ containing $i$.
Fishburn~\cite{Fishburn,FishburnReview} conjectured that, for every dual intersecting family, there exists $i\in[n]$ such that $|\calF^L(i)|\ge|\calF(i)|$. 
He proved that this is true when $\min_{i\in[n]}|\calF(i)|\le n$ or $\min_{i\in[n]}|\calF(i)|<8$~\cite{Fishburn}.
In this paper, we completely prove the conjecture.

\begin{theorem}\label{thm:fishburn}
Suppose that $n\ge1$ and $\calF\subseteq2^{[n]}$ is dual intersecting.
Then there is an $i\in[n]$ such that
\[
|\calF^L(i)|\ge|\calF(i)|.
\]
\end{theorem}

Fishburn noted that this problem is closely related to Chv\'atal's conjecture~\cite{FishburnReview}.
A family $\calD\subseteq 2^{[n]}$ is a \emph{downset} if $A\in\calD$ and $B\subseteq A$ imply $B\in\calD$.
Chv\'atal's conjecture~\cite{Chvatal74} asserts that every downset $\calD$ has an intersecting subfamily of maximum cardinality of the form $\calD(i)$ for some $i\in[n]$.
Kleitman~\cite{Kleitman79} proposed a stronger weighted version.
Friedgut, Kahn, Kalai, and Keller studied related correlation inequalities~\cite{FKKK}.
Chang, Liu, and Liu~\cite{CLL} recently proved the weighted conjecture.
Keevash~\cite{Keevash} proved Kahn's flow conjecture, a stronger form of Chv\'atal's conjecture.
Ellis, Filmus, and Friedgut~\cite{EFF} subsequently gave a short spectral proof of Chv\'atal's conjecture and a strengthening concerning projection packings.

We will use the explicit weighted star inequality in~\cite[Proposition~5.3]{CLL} to derive the bound in Lemma~\ref{lem:normalized}.
We then prove the minimum-influence inequality in Theorem~\ref{thm:witness} and deduce Theorem~\ref{thm:fishburn}.

The paper is organized as follows.
Section~\ref{sec:pre} collects the preliminaries and the weighted star inequality. 
In Section~\ref{sec:main}, we prove the normalized star bound and the minimum-influence inequality, and then deduce Theorem~\ref{thm:fishburn}.

\section{Preliminaries}
\label{sec:pre}
A family $\calU\subseteq2^{[n]}$ is an \emph{upset} if $A\in\calU$ and $A\subseteq B\subseteq[n]$ imply $B\in\calU$.
A \emph{maximal intersecting} family is an intersecting family not properly contained in another intersecting family.
A family $\mathcal A\subseteq2^{[n]}$ is called \emph{antipodal} if it contains exactly one of $A$ and $[n]\setminus A$ for every $A\subseteq [n]$.
We use the following characterization.

\begin{lemma}[\cite{CLL}, Proposition~4.1]\label{lem:maximal}
A family $\calB\subseteq2^{[n]}$ is maximal intersecting if and only if
it is an upset and antipodal. In particular, every maximal intersecting family has $2^{n-1}$ members.
\end{lemma}

For $\calA\subseteq2^{[n]}$, write
\[
 \check{\calA}=\left\{[n]\setminus A:A\in\calA\right\}
\]
for the family of complements of its members.
For any set $A$, let $\mathbf{1}_A$ denote its indicator function, equal to $1$ on $A$ and $0$ elsewhere.
We identify $2^{[n]}$ with $\{0,1\}^n$ via $T\mapsto\mathbf{1}_T$, and write $f(T)$ for $f(\mathbf{1}_T)$.

For a Boolean function $f:\{0,1\}^n\to\{0,1\}$ and a fixed $i\in[n]$, define its influence at $i$ by
\[
\Inf_i(f)
=\frac{1}{2^{n-1}}
\sum_{T\subseteq[n]\setminus\{i\}}
\bigl|f(T\cup\{i\})-f(T)\bigr|.
\]
Since \(f\) is Boolean-valued,
$
|f(T\cup\{i\})-f(T)|
=\mathbf{1}_{\{f(T\cup\{i\})\ne f(T)\}}.
$
Consequently,
\[
\Inf_i(f)
=\Pr\bigl(f(T\cup\{i\})\ne f(T)\bigr),
\]
where \(T\) is chosen uniformly at random from \(2^{[n]\setminus\{i\}}\). Thus, \(\Inf_i(f)\) is the probability that flipping the \(i\)-th coordinate changes the value of \(f\).

For a fixed \(S\subseteq[n]\), the Fourier coefficient of \(f\) is
\[
\widehat f(S)=\frac{1}{2^n}\sum_{T\subseteq[n]}f(T)(-1)^{|S\cap T|}.
\]
The Fourier coefficient \(\widehat f(S)\) measures the contribution of the parity pattern on \(S\) to \(f\).
In particular, when \(S=\varnothing\), we have 
\[
\widehat f(\varnothing)=\frac{1}{2^n}\sum_{T\subseteq[n]}f(T)=\mathbb E[f],
\] 
where $\mathbb E[f]$ is the mean of \(f\).
By the standard Fourier formula for influences
(see, e.g., \cite[Equation~(10)]{CLL}),
\begin{equation}\label{eq:fourier-influence}
 \Inf_i(f)=4\sum_{S\ni i}\widehat f(S)^2.
\end{equation}

We use the following form of the weighted star inequality.
Fix a total order $\prec$ on $[n]$, and write $\max_\prec S$ for the last element of a nonempty set $S\subseteq[n]$.

\begin{proposition}[\cite{CLL}, Proposition~5.3]
\label{prop:CLL}
Let $\calB\subseteq2^{[n]}$ be maximal intersecting, and set $f=\mathbf{1}_{\calB}$.
For $i\in[n]$, define
\[
 \lambda_i
 =4\sum_{\substack{\varnothing\ne S\subseteq[n]\\
\max_\prec S=i}}
\widehat f(S)^2.
\]
Then $\lambda_i\ge0$ for every $i\in[n]$ and $\sum_{i=1}^n\lambda_i=1$. Moreover, if $\omega:2^{[n]}\to\mathbb R_{\ge0}$ satisfies
\[
A\subseteq B\quad\Longrightarrow\quad
\omega(A)\ge\omega(B),
\] 
then
\begin{equation}\label{eq:weighted-star}
\sum_{S\in\calB}\omega(S)
\le
\sum_{i=1}^n\lambda_i
\sum_{S\ni i}\omega(S)
\le
\max_{i\in[n]}
\sum_{S\ni i}\omega(S).
\end{equation}
\end{proposition}

\begin{remark}
The indicator function of any downset is nonnegative and decreasing, so it is an admissible weight in~\eqref{eq:weighted-star}.
If $z$ is last in the order $\prec$, then $\max_\prec S=z$ if and only if $z\in S$. 
Hence, by
\eqref{eq:fourier-influence},
\begin{equation}\label{eq:last-coefficient}
 \lambda_z
 =4\sum_{S\ni z}\widehat f(S)^2
 =\Inf_z(f).
\end{equation}
\end{remark}



The coefficients \(\lambda_i\) depend only on \(\mathcal B\) and the chosen order. Since they are nonnegative and sum to \(1\), the middle expression in \eqref{eq:weighted-star} is a weighted average of the total weights of the coordinate stars. Thus, \eqref{eq:weighted-star} bounds the weight of \(\mathcal B\) by this weighted average, and hence by the largest star weight.
A useful choice is \(\omega=\mathbf{1}_{\mathcal D}\), where \(\mathcal D\) is a downset. Indeed, \(A\subseteq B\) implies \(\mathbf{1}_{\mathcal D}(A)\ge \mathbf{1}_{\mathcal D}(B)\), so this weight satisfies the required monotonicity condition. Thus, the weighted sums become cardinalities, and \eqref{eq:weighted-star} gives
\[
|\mathcal B\cap\mathcal D|
\le \sum_{i=1}^{n}\lambda_i|\mathcal D(i)|
\le \max_{i\in[n]}|\mathcal D(i)|.
\]

\section{Proof of the main theorem}
\label{sec:main}
For $\calA\subseteq2^{[n]}$, define its Boolean dual by
\[
 \calA^\dagger
 =\{T\subseteq[n]:[n]\setminus T\notin\calA\}
 =2^{[n]}\setminus\check{\calA}.
\]
The following is immediate by definition.
\begin{lemma}\label{basiclemma}
Suppose that $\calA$ is an upset.
Then the Boolean dual $\calA^\dagger$ is an upset.
If $\calA$ is intersecting, then $\calA\subseteq\calA^\dagger$.
\end{lemma}

We now derive an inequality for intersecting upsets.
This inequality was stated as Conjecture~3.7 in~\cite{FKKK}.
It follows from Keevash's theorem on Kahn's flow conjecture~\cite[Theorem~1.1]{Keevash} and can also be deduced from the argument of Ellis, Filmus, and Friedgut~\cite[Section~2]{EFF}.
For completeness, we give a proof based on the weighted star
inequality of Chang, Liu, and Liu~\cite[Proposition~5.3]{CLL}.

\begin{lemma}\label{lem:normalized}
Let $n\ge1$. 
Suppose that $\calH\subseteq2^{[n]}$ is a nonempty intersecting upset and $\calE\subseteq2^{[n]}$ is a downset. Then
\begin{equation}\label{eq:normalized}
|\calH\cap\calE|
\le \frac{|\calH|}{2^{n-1}}
\max_{i\in[n]}|\calE(i)|.
\end{equation}
\end{lemma}

\begin{proof}
Define
\[
\mathcal K
=
\mathcal H
\cup
\bigl\{T\cup\{n+1\}:T\in\mathcal H^\dagger\bigr\}
\subseteq 2^{[n+1]}.
\]
By Lemma~\ref{basiclemma}, both $\mathcal H$ and $\mathcal H^\dagger$ are
upsets, and $\mathcal H\subseteq\mathcal H^\dagger$.
We will apply Proposition~\ref{prop:CLL} to $\mathcal K$.
We show that $\mathcal K$ is an upset. 

Let $A\in\mathcal K$
and let $A\subseteq B\subseteq[n+1]$.
If $n+1\notin B$, then neither $A$ nor $B$ contains $n+1$.
Thus $A\in\mathcal H$, and the upward closure of $\mathcal H$
gives $B\in\mathcal H\subseteq\mathcal K$.
If $n+1\in A$, then
\[
A\setminus\{n+1\}\in\mathcal H^\dagger,
\qquad
A\setminus\{n+1\}\subseteq B\setminus\{n+1\}.
\]
Since $\mathcal H^\dagger$ is an upset, we obtain
$B\setminus\{n+1\}\in\mathcal H^\dagger$, and hence
$B\in\mathcal K$.
Suppose that $n+1\notin A$ but $n+1\in B$.
Then
\[
A\in\mathcal H\subseteq\mathcal H^\dagger,
\qquad
A\subseteq B\setminus\{n+1\}.
\]
Again, the upward closure of $\mathcal H^\dagger$ gives
$B\setminus\{n+1\}\in\mathcal H^\dagger$, so $B\in\mathcal K$.
Therefore, $\mathcal K$ is an upset.
For every $T\subseteq[n]$, the definition of $\calH^\dagger$ gives
\[
T\in\calK
\iff T\in\calH  \iff [n]\setminus T\notin\calH^\dagger
 \iff ([n]\setminus T)\cup\{n+1\}\notin\calK.
\]
Since $T$ and $([n]\setminus T)\cup\{n+1\}$ are complements in $[n+1]$, $\calK$ is antipodal. 
Thus Lemma~\ref{lem:maximal} implies that $\calK$ is maximal intersecting.

Apply Proposition~\ref{prop:CLL} to  $\mathcal K\subseteq 2^{[n+1]}$.
Choose a total order $\prec$ on $[n+1]$ in which $n+1$
comes last, and put $f=\mathbf{1}_{\mathcal K}$.
Let $\lambda_1,\ldots,\lambda_{n+1}$ be the coefficients
associated with $f$ and this order.
In particular, Proposition~\ref{prop:CLL} gives
\[
\lambda_i\ge 0
\quad\text{and}\quad
\sum_{i=1}^{n+1}\lambda_i=1.
\]

Since $n+1$ is last in the order, the condition
$\max_{\prec}S=n+1$ is equivalent to $n+1\in S$.
Thus,
\[
\lambda_{n+1}
=
4\sum_{\substack{S\subseteq[n+1]\\ n+1\in S}}
\widehat f(S)^2
=
\Inf_{n+1}(f).
\]
To compute $\Inf_{n+1}(f)$, the definition
of $\mathcal K$ gives, for every $T\subseteq[n]$,
\[
f(T)=\mathbf{1}_{\mathcal H}(T),
\qquad
f(T\cup\{n+1\})=\mathbf{1}_{\mathcal H^\dagger}(T).
\]
Since $\mathcal H\subseteq\mathcal H^\dagger$, these
two values differ precisely when
$T\in\mathcal H^\dagger\setminus\mathcal H$.
Then we have 
\[
\begin{aligned}
\lambda_{n+1}
=\Inf_{n+1}(f)
=\frac{1}{2^n}\sum_{T\subseteq[n]}
  \bigl|f(T\cup\{n+1\})-f(T)\bigr|=\frac{|\mathcal H^\dagger\setminus\mathcal H|}{2^n}.
\end{aligned}
\]

Since $\calH^\dagger=2^{[n]}\setminus\check{\calH}$ and $|\check{\calH}|=|\calH|$,
we have 
\[
 |\calH^\dagger|
 =|2^{[n]}|-|\check{\calH}|
 =2^n-|\calH|,\qquad |\calH^\dagger\setminus\calH|
 =|\calH^\dagger|-|\calH|
 =2^n-2|\calH|.
\]
Therefore
\[
 \lambda_{n+1}
 =\frac{|\calH^\dagger\setminus\calH|}{2^n}
 =1-\frac{|\calH|}{2^{n-1}}.
\]
By Proposition~\ref{prop:CLL}, $\sum_{i=1}^{n+1}\lambda_i=1$. Then
\begin{equation}\label{eq:lambda-base}
 \sum_{i=1}^{n}\lambda_i
 =\frac{|\calH|}{2^{n-1}}.
\end{equation}

For the downset $\calE$,
define $\omega:2^{[n+1]}\to\mathbb R_{\ge0}$ by
\[
 \omega(T)=
 \begin{cases}
  \mathbf{1}_{\calE}(T),&n+1\notin T,\\
  0,&n+1\in T.
 \end{cases}
\]
In particular, every set containing $n+1$ has weight zero.
This weight is nonnegative and decreasing because $\calE$ is a downset.
Indeed, let
$A\subseteq B\subseteq[n+1]$. If $\omega(B)=0$, then
$\omega(A)\ge\omega(B)$ automatically. If $\omega(B)=1$,
then $B\in\calE$, and the downward closure of $\calE$
implies $A\in\calE$. Hence $\omega(A)=\omega(B)=1$.

Since $\omega$ is the indicator
of $\calE$, the total weight of $\calK$ is
$|\calK\cap\calE|$. Moreover, every member of $\calE$
lies in $2^{[n]}$, where $\calK$ agrees with $\calH$.
Therefore,
\[
\sum_{T\in\calK}\omega(T)
=
|\calK\cap\calE|
=
|\calH\cap\calE|.
\]
For each $i\in[n]$, the total weight of the $i$-star
counts precisely the members of $\calE$ that contain $i$.
Thus,
\[
\sum_{\substack{T\subseteq[n+1]\\i\in T}}\omega(T)
=
|\{T\in\calE:i\in T\}|
=
|\calE(i)|.
\]
The $(n+1)$-star has total weight zero, since no member
of $\calE$ contains $n+1$.

Substituting these quantities into
\eqref{eq:weighted-star}, we obtain
\[
|\calH\cap\calE|
\le
\sum_{i=1}^{n}\lambda_i|\calE(i)|
+\lambda_{n+1}\cdot0
=
\sum_{i=1}^{n}\lambda_i|\calE(i)|.
\]

Finally, since each $\lambda_i$ is nonnegative, replacing
every $|\calE(i)|$ by their maximum can only increase
the right-hand side. Using \eqref{eq:lambda-base}, we get
\[
\begin{aligned}
|\calH\cap\calE|
&\le
\left(\sum_{i=1}^{n}\lambda_i\right)
\max_{i\in[n]}|\calE(i)|=
\frac{|\calH|}{2^{n-1}}
\max_{i\in[n]}|\calE(i)|.
\end{aligned}
\]
This completes the proof.
\end{proof}

We now present the following inequality used in the proof of  Theorem~\ref{thm:fishburn}.

\begin{theorem}
\label{thm:witness}
Let $n\ge2$, and let $\calD\subseteq2^{[n]}$ be a downset such that $A\cup A'\ne[n]$ for all $A,A'\in\calD$.
Let $\calB\subseteq2^{[n]}$ be maximal intersecting, and choose $i\in[n]$ with $\Inf_i(\mathbf{1}_{\calB})=\min_{j\in[n]}\Inf_j(\mathbf{1}_{\calB})$.
Then
\[
 2\,|\calB\cap\calD(i)|\le|\calD(i)|.
\]
\end{theorem}

\begin{proof}
By relabelling the ground set, we may  assume that $i=n$, and
\begin{equation}\label{mini-condition}
\Inf_n(\mathbf{1}_{\mathcal B})
\le \Inf_j(\mathbf{1}_{\mathcal B}),
\qquad j\in[n-1].
\end{equation}

\noindent\textbf{Step 1. Transfer the counting problem to $[n-1]$.}
Define
\[
 \calE=\{T\subseteq[n-1]:T\cup\{n\}\in\calD\},
 \qquad
 \calH=\calB\cap2^{[n-1]},
 \qquad
 \calC=2^{[n-1]}\setminus\calH.
\]
Let $\check{\calE}$ denote the family obtained by
taking complements of the members of $\calE$
relative to $[n-1]$:
\[
\check{\calE}
=\{[n-1]\setminus T:T\in\calE\}.
\]
By Lemma~\ref{lem:maximal}, $\calB$ contains exactly
one member of each complementary pair in $2^{[n]}$.
Thus, for every $T\subseteq[n-1]$,
\begin{equation}\label{eq:B-upper-section}
T\cup\{n\}\in\calB
\iff [n-1]\setminus T\notin\calB\iff [n-1]\setminus T\notin\calH\iff [n-1]\setminus T\in\calC.  
\end{equation}
Also,
\[
 T\cup\{n\}\in\calD(n)\iff T\in\calE\iff [n-1]\setminus T\in\check{\calE}.
\]
Under the bijection $T\cup\{n\}\mapsto[n-1]\setminus T$, these equivalences yield
\begin{equation}\label{eq:witness-reduction}
|\calB\cap\calD(n)|=|\check{\calE}\cap\calC|,
 \qquad
 |\calD(n)|=|\check{\calE}|.
\end{equation}
It suffices to prove that 
\begin{equation}
2\,|\check{\calE}\cap\calC|\le |\check{\calE}|.
\end{equation}

\medskip
\noindent\textbf{Step 2. Verify the hypotheses and apply
Lemma~\ref{lem:normalized}.}
If $\calE=\varnothing$, then $\calD(n)=\varnothing$, and the conclusion follows. 
Assume that $\calE\ne\varnothing$.

Since $\calD$ is a downset, $\calE$ is a downset in $2^{[n-1]}$, and therefore $\check{\calE}$ is an upset.
For $S,T\in\calE$, the hypothesis on $\calD$ gives
\[
(S\cup\{n\})\cup(T\cup\{n\})\ne[n].
\]
Consequently,
\[
([n-1]\setminus S)\cap([n-1]\setminus T)
=[n-1]\setminus(S\cup T)\ne\varnothing.
\]
Thus $\check{\calE}$ is a nonempty intersecting upset.

By Lemma~\ref{lem:maximal}, $\calB$ is an upset.
Hence $\calH$ is an upset in $2^{[n-1]}$, and $\calC$ is a downset.
Applying Lemma~\ref{lem:normalized} to $\check{\calE}$ and $\calC$ on the ground set $[n-1]$, we obtain
\begin{equation}\label{eq:witness-normalized}
 |\check{\calE}\cap\calC|
 \le\frac{|\check{\calE}|}{2^{n-2}}
       \max_{j\in[n-1]}|\calC(j)|.
\end{equation}
\medskip
\noindent\textbf{Step 3. Bound the sizes of $\calC(j)$.}
We will show that $|\calC(j)|\le 2^{n-3}$
for every $j\in[n-1]$.

Put $f=\mathbf{1}_{\calB}$. 
Since $\calB$ is an upset, $f(T\cup\{k\})-f(T)\ge0$ for every $T\subseteq[n]\setminus\{k\}$. 
Hence
\[
\begin{aligned}
 \Inf_k(f)
 =\frac{1}{2^{n-1}}
   \sum_{T\subseteq[n]\setminus\{k\}}
       \bigl(f(T\cup\{k\})-f(T)\bigr)
 =\frac{|\calB(k)|
         -|\{A\in\calB:k\notin A\}|}{2^{n-1}}.
\end{aligned}
\]
Since $\calB$ contains exactly one set from each complementary pair, $|\calB|=2^{n-1}$. 
Therefore
\begin{equation}\label{eq:influence-star}
 \Inf_k(f)
 =\frac{2|\calB(k)|-2^{n-1}}{2^{n-1}}
 =\frac{|\calB(k)|}{2^{n-2}}-1,
 \qquad k\in[n].
\end{equation}
By \eqref{eq:B-upper-section}, the sets in $\calB(n)$ correspond bijectively to $\calC$. 
Hence
\[
 |\calB(n)|=|\calC|=2^{n-1}-|\calH|,
 \qquad
 \Inf_n(f)=1-\frac{|\calH|}{2^{n-2}}.
\]

Now fix $j\in[n-1]$. 
The members of $\calB(j)$ that do not contain $n$ are counted by $|\calH(j)|$.
By \eqref{eq:B-upper-section}, the members of $\calB(j)$ that contain $n$ correspond to the sets in $\calC$ that do not contain $j$. 
Therefore
\[
\begin{aligned}
 |\calB(j)|
 &=|\calH(j)|+|\calC|-|\calC(j)|\\
 &=|\calH(j)|
   +(2^{n-1}-|\calH|)
   -(2^{n-2}-|\calH(j)|)\\
 &=2^{n-2}-|\calH|+2|\calH(j)|.
\end{aligned}
\]
Using \eqref{eq:influence-star} again gives
\[
 \Inf_j(f)
 =\frac{2|\calH(j)|-|\calH|}{2^{n-2}}.
\]

By \eqref{mini-condition}, $\Inf_n(f)\le\Inf_j(f)$.
The formulas above give
\[
 1-\frac{|\calH|}{2^{n-2}}
 \le
 \frac{2|\calH(j)|-|\calH|}{2^{n-2}},
\]
so $|\calH(j)|\ge2^{n-3}$. 
Since $\calH(j)$ and $\calC(j)$ partition the subsets of $[n-1]$
containing $j$,
\[
 |\calC(j)|
 =2^{n-2}-|\calH(j)|
 \le2^{n-3},
 \qquad j\in[n-1].
\]
Thus \eqref{eq:witness-normalized} yields
\[
 |\check{\calE}\cap\calC|
 \le\frac{|\check{\calE}|}{2^{n-2}}
       \max_{j\in[n-1]}|\calC(j)|
 \le\frac{|\check{\calE}|}{2}.
\]
Finally, \eqref{eq:witness-reduction} gives
\[
 2|\calB\cap\calD(n)|
 =2|\check{\calE}\cap\calC|
 \le|\check{\calE}|
 =|\calD(n)|,
\]
as required.
\end{proof}

\begin{proof}[Proof of Theorem~\ref{thm:fishburn}]
If $\calF=\varnothing$, then every $i\in[n]$ satisfies the conclusion. 
Suppose that $\calF\ne\varnothing$, and choose $A\in\calF$.
Since $\calF$ is dual intersecting, applying its defining conditions
to $A$ paired with itself gives
$\varnothing\subsetneq A\subsetneq[n]$.
In particular, $n\ge2$.

Set $\calD=\{T\subseteq[n]:T\subseteq A
\text{ for some }A\in\calF\}$.
The family $\calD$ is a downset.
For $X,Y\in\calD$, choose $A,A'\in\calF$ such that $X\subseteq A$ and
$Y\subseteq A'$.
Since $\calF$ is dual intersecting, $A\cup A'\ne[n]$.
As $X\cup Y\subseteq A\cup A'$, we obtain $X\cup Y\ne[n]$.

Extend $\calF$ to a maximal intersecting family $\calB$, and choose $i\in[n]$ for which $\Inf_i(\mathbf{1}_{\calB})$ is minimum. 
Since $\calF(i)\subseteq\calB\cap\calD(i)$, Theorem~\ref{thm:witness} gives
\[
 2|\calF(i)|
 \le 2|\calB\cap\calD(i)|
 \le|\calD(i)|.
\]
Finally, $\calD=\calF\sqcup\calF^L$, so $|\calD(i)|=|\calF(i)|+|\calF^L(i)|$.
Thus $|\calF(i)|\le|\calF^L(i)|$.
\end{proof}

\section*{Statement on AI use}
This proof was found by GPT-6 Astra, following an approach suggested by the author and using the weighted star inequality of Chang, Liu, and Liu~\cite{CLL}. 
The author simplified and rewrote the proof and takes full responsibility for the results presented in this paper.

\end{document}